\documentclass[12pt]{amsart}
\usepackage{amssymb} 

\newtheorem{thm}{Theorem}[section]
\newtheorem{cor}[thm]{Corollary}
\newtheorem{lem}[thm]{Lemma}
\newtheorem{prop}[thm]{Proposition}

\newtheorem{conj}[thm]{Conjecture}
\theoremstyle{definition}

\newtheorem{rem}[thm]{Remark}

\newtheorem{que}[thm]{Question}

\numberwithin{equation}{section}

\usepackage[colorlinks]{hyperref}

\begin{document}
\title[On a  bijection between   a finite group\ldots    ]
{On a  bijection between   a finite group    and   cyclic group}%
\author[    ]{  Mohsen Amiri}%
\address{ Departamento de  Matem\'{a}tica, Universidade Federal do Amazonas.}%
\email{m.amiri77@gmail.com}%
\email{}
\thanks{}
\thanks{}


\begin{abstract}
 We show that for any finite group $G$ of order $n$ there exists  a bijection $f$ from $G$ onto the cyclic group $C_{n}$ such that $o(x)$ divides $o(f(x))$ for all $x\in G$. This confirms Problem 18.1 in  \cite{Kh}.
\end{abstract}

\maketitle


\section{\bf Introduction}

 For a group $G$,  the set of orders of all
elements in $G$ is denoted by $\omega(G)$ which has been recently called the spectrum of $G$. 
One of the most interesting concepts in Finite Group Theory which has
recently attracted several researchers is the problem of characterizing finite
groups by element orders. 
For example, A finite group $G$ is said to be recognizable by spectrum, i.e., by the set of element orders,
if every finite group $H$ having the same spectrum as $G$ is isomorphic to $G$.
In 1987, in his letter to J. Thompson, W. Shi conjectured that every finite simple group
is recognizable in the class of finite groups by its order and the set of element orders. In his
answer, J. Thompson highly appreciated this conjecture and put forward another one: every
finite nonabelian simple group is recognizable in the class of finite groups with trivial center
by the set of sizes of conjugacy classes. As a result of efforts by many mathematicians started in \cite{97} and \cite{18}, the validity of
 both conjecture  was eventually established.
For any $m\in \omega(G)$ let $s(m)$ be    the size of the subset $\{g\in G: o(g)=m\}$ of $G$.
In \cite{jaf}, H. Amiri, S.M.J. Amiri and  M. Isaacs studied the spectrum of a finite group $G$ from a different point of view: they introduced the function  $\psi(G)=\sum_{g\in G}o(g)=\sum_{m\in \omega(G)}m\cdot s(m)$, and they proved in the set of all finite groups of order $n$, $\psi(C_n)$ is maximal. Starting from this result, many   mathematicians  have recently studied the function $\psi(G)$ and its relations with
the structure of $G$, also, in \cite{21} and \cite{22}, the authors introduced  some solubility criteria for a certain group $G$ with respect its valor  $\psi(G)$.   M. Isaacs in \cite{lad},  ask a question about  the spectrum of a finite group $G$ from a different point of view:

\begin{que}\label{is}(see Problem 18.1 in  \cite{Kh})
{\it Does there necessarily exist a bijection $f$ from $G$ onto a cyclic group of order $n$ such that for each element $x\in G$, the order of $x$ divides the order of $f(x)$}?
\end{que}

An affirmation in the  special case of a solvable group $G$  has already been settled  by F. Ladisch \cite{lad}. 
Note that if $n=8m$, then  we can not  replace the cyclic group $C_n$ with other abelian groups  in Question \ref{is}. For example,  in \cite{mohsen}, authors show that if $n=8m$, then 
there is no any bijection $f$ from
$C_m\times Q_8$ onto  $C_{4m}\times C_2$ such that $o(x)\mid o(f(x))$. Hence, the cyclic group $C_n$  cannot be replaced by other non-cyclic abelian  groups of order $n$.
The main objective of this article is   to give an affirmative answer to  Question \ref{is}.

For any two positive integers $n$  and  $p$, the biggest divisor of $n$ which is co-prime to $p$ is denote by $n_{p'}$ and $n/n_{p'}$ is denoted by $n_p$. Also, the set of all prime divisors of $n$ is denoted by $\pi(n)$. For any element $x$ of a group $G$, the symbol  $[x]$ denote  the set of all  generators of the cyclic subgroup $\langle x\rangle$ of $G$. Let $B$ a non-empty subset of a group $G$ of order $n$, and let $d$ a divisor of $n$. 
We denote by $Sol(B,d,G)$ to be the set of all $x\in G$ such that $x^m\in B^G=\{b^g: b\in B, g\in G\}$ for some divisor $m$ of $d$, and  the set of all elements of $G$ of order $d$ is denoted by $B_d$.  Also, the unique  cyclic subgroup of order $d$ of the cyclic group of order $n$ is denoted by $C_{d}$.  

{\bf Proof Outline:}
Let $G$ be a finite group of order $n$, and let $d$ and $m$ be two co-prime divisors of $n$. The   strategy in our proof  start
with a find an injection from a $Sol(B_m,d,G)$ (with $m>1$ and $gcd(m,d)=1)$ of $G$   into   $C_n$, and then build
injective maps on larger subsets. The crucial tools here  are  Lemma \ref{jai} and Proposition \ref{dis},
which allows to change part of the injection already constructed and
move subsets of $C_n$, so that the changed injection can
be extended to a larger subset of
$C_n$. Next, we use this injection, the  equality
$Sol(1,d,G)=Sol(B_{d_p},d_{p'},G)\cup Sol(1,d/p,G)$ for the smallest prime divisor $p$ of $d$, and induction, to construct an injection from 
$Sol(1,d,G)$ to $C_n$.

In what follows, we adopt the notations  established in the Isaacs' book on finite groups  \cite{I}.


\section{Results}\label{sec2}

 We start with the following  
fundamental theorem proved by Frobenius \cite{fr}, more than hundred years ago, in 1895:

\begin{thm}
If $d$ is a divisor of the order of a finite group $G$, then the number of solutions of
$x^d=1$ in $G$ is a multiple of $d$.
\end{thm}
 
 The part (v) of the following proposition generalized
  Theorem 3, of \cite{Wei}.

\begin{prop}\label{dis}
Let $G$ be a finite group, and let $d$ and $t$ be divisors of $|G|$ such that $gcd(t,d)=1$.
For any $a\in B_t$ let 
$$S_a=\{ba:b\in Sol(1,d,C_G(a))\}$$ and
$D_a:= Sol(1,d,\frac{C_G(a)}{\langle a\rangle })$. 

(i) For any $a_1,a_2\in B_t$, if $a_1\neq a_2$, then $S_{a_1}\cap S_{a_2}=\varnothing$.

(ii) For any $a\in B_t$ the function $f$ from $S_a$ into $D_a$ by $f(ba)=b\langle a\rangle$ is a bijection such that $o(ba)\mid o(f(ba))t$. 

(iii) For any $a_1,a_2\in B_t$, if $a_1\neq a_2$, then for all subgroups $H$ of $G$, we have $Sol([a_1],d,H)\cap Sol([a_2],d,H)=\varnothing$ or $Sol([a_1],d,H)=Sol([a_2],d,H)$.
 
 (iv) For any   $a_1\neq a_2 \in B_t$,  we have  $Sol([a_1],d,C_G(a_1))\cap Sol([a_2],d,C_G(a_2))=\varnothing$ or $Sol([a_1],d,C_G(a_1))=Sol([a_2],d,C_G(a_2))$.

(v) For any $a\in B_t$,   
  $$ lcm(\varphi(t), d ) \mid |Sol([a],d,G)|. $$

\end{prop}
\begin{proof}
(i) Let $b_1a_1=b_2a_2\in S_{a_1}\cap S_{a_2}$ where $a_1\neq a_2\in B_t$. Clearly,   $o(b_1)=o(b_2)$, because  $b_1a_1=a_1b_1$ and $a_2b_2=b_2a_2$. Also, we have
$a_{1}^m=a_2^m$ where $m=o(b_1)$.
Since $gcd(t,m)=1$, we have $\langle a_1\rangle=\langle a_2\rangle$, so 
$a_1=a_2^j$ for some integer $j$. Then 
$b_1=b_2a_{2}^{1-j}$. Since $gcd(t,m)=1$, we have $1-j=0$, showing that 
$a_1=a_2$, which is a contradiction.

(ii) Clearly, $f$ is a surjection map.
If $b\langle a\rangle=y\langle a\rangle\in D_a$, then 
$ba=ya^i$ for some integer $i$. So by a similar argument as the case (i), we have $b=y$, showing that $f$ is an injection.
Clearly, $o(ba)=o(f(ba))o(a)$.

 (iii)  It is easy to show that for any $a\in B_t$, we have $Sol([a],d,H)=\bigcup_{g\in G} \{ba^g:b\in Sol(1,d,C_H(a^g))\}$. Let $a_1,a_2\in B_t$.  Suppose that $b_1a_1=b_2a_2\in   Sol([a_1],d,H)\cap Sol([a_2],d,H)$ where $b_1\in C_G(a_1)$ and $b_2\in C_G(a_2)$ such that $b_1^d=b_2^d=1$.
 Then by a the same argument as the proof of the case (i) $a_1=a_2$.
 Therefore $Sol([a_1],d,H)=Sol([a_2],d,H)$.
 
 (iv) Let $b_1a_1=b_2a_2 \in Sol([a_1],d,C_G(a_1))\cap Sol([a_2],d,C_G(a_2))$ where $b_1\in C_G(a_1)$ and $b_2\in C_G(a_2)$ such that $b_1^d=b_2^d=1$. Then by a the same argument as the proof of the case (i) $a_1=a_2$, and so $C_G(a_1)=C_G(a_2).$ It follows from the case (v) that $Sol([a_1],d,C_G(a_1))= Sol([a_2],d,C_G(a_2))$.

(v) 
Let $a\in B_t$.  
It is clear that $\bigcup_{g\in G}\bigcup_{a^i\in [a]}S_{a^{ig}} \subseteq Sol([a],d,G).$
Let $x\in Sol([a],d,G)$. Then there exist $g\in G$, a divisor $m$ of $d$ and $a^i\in [a]$ such that 
$x^d=(a^i)^g$. Hence $x\in C_G(a^g)$, and so
$x^{t}\in C_G(a^g).$ Therefore $x^{t}(a^{i})^g\in S_{(a^i)^g}$, showing that $x\in \bigcup_{g\in G}\bigcup_{a^i\in [a]}S_{a^{ig}}$. Consequently, $$\bigcup_{g\in G}\bigcup_{a^i\in [a]}S_{a^{ig}} = Sol([a],d,G).$$  
Therefore $ |a^G||S_a|\mid |Sol([a],d,G)|$. 
Let $k=gcd(|C_G(a)|,d)$.
Define $f$ from $S_a$ into $Sol(1,k,\frac{C_G(a)}{\langle a\rangle})$ by $f(ba)=b\langle a\rangle$.
From Case (ii),  $f$ is a bijection such that $o(ba)\mid o(b\langle a\rangle)t$ for all
$ba\in S_a$.  By Frobenius Theorem $k\mid |Sol(1,k,\frac{C_G(a)}{\langle a\rangle})|$, so $k\mid |Sol([a],k,C_G(a))|.$
Since $ \frac{d}{k}\mid   |a^G|$ , we have 
$ d \mid |Sol([a],d,G)|$. 

Define an equivalency relation $\sim$ on $Sol([a],d,G)$ as follows:
 Let $x,y\in Sol([a],d,G)$.
We say $x\sim y$ whenever $\langle x \rangle=\langle y\rangle$.
Then $\sim $ is an equivalency relation on $Sol([a],d,G)$.
Let $c(x)$ be the equivalency class of $x\in Sol([a],d,G)$.
Since $c(x)=c(y)$ if and only if 
$y=x^i$ for some $x^i\in [x]$, we deduce that $\varphi(o(x))= |[x]|$.    Writing $Sol([a],d,G)$ as a disjoint union of its equivalence classes, it follows that
$\varphi(t)\mid |Sol([a],d,G)|$.
Consequently,  $ lcm(\varphi(t), d ) \mid |Sol([a],d,G)|.$  
 
\end{proof}
\begin{rem}
  Note that by a similar arguements as the proof of Proposition \ref{dis} (v), we may prove that 
  $d\mid |Sol([a],d,G)|$ for all $a\in G$ such that $gcd(o(a),d)=1.$
But  in general $   d \nmid |Sol(a,d,G)|$. 
For example if $G=C_{15}=\langle h\rangle$, then 
$Sol(h^3,3,G)=\{h,h^3,h^6,h^{11}\}$. Also,  in previous theorem the equality   $Sol(a_1,d,G)\cap Sol(a_2,d,G)=\varnothing$  for any $a_1\neq a_2\in Q$, is not true.
\end{rem}

\begin{lem}\label{ja}
Let $G$ be a finite group of order $n$, and  $d$ be a divisor of $exp(G)$. Let $p$ be the smallest prime divisor of $d$. Suppose that there exists a subset $Y$ of  $C_n$      with the  following properties:

(i) $|Sol(1,d/p,G)|=|Y|$.

(ii) If $y\in Y$  and $x\in C_n\setminus Y$, then 
  $o(yC_{d/p})_p\leq o(xC_{d/p})_p$.
 
 Then
  \begin{equation}\label{ppp22}
|Sol(1,d,G)\setminus Sol(1,d/p,G)|\leq |\{x\in C_n\setminus Y: d_p\mid o(x)\}|.
\end{equation}

 \end{lem}
 \begin{proof}
 
By Frobenius Theorem, $|Sol(1, n_{p'}\cdot p^{\alpha-1},G)|= n_{p'}\cdot p^{\alpha-1}e$ for some integer $e\geq 1$.
Let $Y=\{y_1,\ldots,y_r\}$. 
We may assume that $o(y_1C_{d/p})_p\leq o(y_2C_{d/p})_p\leq\ldots\leq o(y_rC_{d/p})_p$.
Let $d_p=p^{\alpha}$.
 If $p\mid o(y_{r}C_{d/p})$, then $p^{\alpha}\mid o(y_r)$.
  If there exists $z\in C_{d/p}\setminus Y$, then  
  $p^{\alpha}\nmid o(z)$, so 
$o(zC_{d/p})_p<o(y_rC_{d/p})_p$, which is a contradiction.
So
  $C_{n_{p'}\cdot p^{\alpha-1}}\subseteq Y$.
It follows that for all $y\in C_n\setminus Y$, we have $p^{\alpha}\mid o(y)$.
Hence, clearly, \ref{ppp22}, is true.

So suppose that $p\nmid o(y_{r}C_{d/p})$.
Then $y_r\in C_{n_{p'}\cdot p^{\alpha-1}}$.
Since for all $u\in C_n\setminus Y$, we have $o(y_{r}C_{d/p})_p\leq o(uC_{d/p})_p$, we conclude that
$Y\subseteq C_{n_{p'}\cdot p^{\alpha-1}}$, so
 $|Y|\leq n_{p'}\cdot p^{\alpha-1}$. 
 If $x\in Sol(1,n_{p'}\cdot p^{\alpha-1},G)$, then $x\not\in Sol(1,d,G)\setminus Sol(1, d/p,G)$.
 It follows that $$|Sol(1,d,G)\setminus Sol(1,d/p,G)|\leq n-n_{p'}\cdot p^{\alpha-1}e.$$ Consequently,
 \begin{eqnarray*}
 |Sol(1,d,G)\setminus Sol(1,d/p,G) |&\leq&n-n_{p'}\cdot p^{\alpha-1}e\\&\leq&  n_{p'}\cdot (n_p\cdot -p^{\alpha-1})\\&\leq& |\{x\in C_n\setminus Y: p^{\alpha}\mid o(x)\}|.
 \end{eqnarray*}

 \end{proof}
 
 \begin{rem}
 Let $G$ be a finite group and let $d$ be a divisor of $exp(G)$.
Let $d=p_1{^{\alpha_1}}\ldots p_k^{\alpha_k}$ be a divisor of $exp(G)$ where $p_1< \ldots<p_k$ are prime numbers.  
We have $Sol(1,d,G)\setminus Sol(1,d/p_1,G)=Sol(B_{p_1^{\alpha_1}},p_1'(d),G)$. Therefore 
$$Sol(1,d,G)=Sol(B_{p_1^{\alpha_1}},d/p_1^{\alpha_1},G)\cup Sol(1,d/p_1,G).$$
If $x\in  Sol(1,d/p_1,G)$, then $p_1^{\alpha_1}\nmid o(x)$ and for all $y\in Sol(B_{p_1^{\alpha_1}},p_1'(d),G)$, we have $p_1^{\alpha_1}\mid o(y)$. It follows that $Sol(B_{p_1^{\alpha_1}},d/p_1^{\alpha_1},G)\cap Sol(1,d/p_1,G)=\varnothing$.
By repeating this process, we deduce that   $$Sol(1,d,G)\setminus \{1\}=\bigcup_{i=1}^k\bigcup_{j=0}^{\alpha_i-1}Sol(B_{p_i^{\alpha_i-j}},\frac{d}{p_1^{\alpha_1}p_2^{\alpha_2}\ldots p_{i}^{\alpha_{i}}},G).$$

\end{rem}

\begin{lem}\label{jai}
Let $d>1$ be a divisor of $n$, and let $p$ be a prime  divisor of $d$.
Let $x,y\in C_n$ where $d_p\mid gcd(o(y),o(x))$. 

(i) There exists  a bijection 
$f$ from 
$xC_{d/p}$ to $yC_{d/p}$ such that 
for $h\in C_{d/p}$ of order $e_h$, we have     $$lcm(d_p, e_h)\mid  o(f(xh)).$$

(ii) If $T$ is a subset of a finite group $G$ with    a bijection $\beta$ from 
$T$ to $xC_{d/p}$ such that $o(t)\mid o(\beta(t)^{o(\beta(t)C_{d})})$ for all $t\in T$, then 
$o(t)\mid o(f(\beta(t))^{o(f(\beta(t))C_{d})})$ for all $t\in T$.

(iii)
If $A$ is a non-empty subset of $C_{d/p}$, the  the restriction of $f$ to $xA$ is a bijection from
 $xA$ to  $yA$.
\end{lem}
\begin{proof}
    
Let $s=o(yC_{d})$ and $t=o(xC_{d})$. We have $(xC_{d/p})^t=p=(yC_{d/p})^s$.
Hence, we may assume that $x$ any $y$ are $p$-elements.
Let $h\in C_{d/p}$ of order $e_h$. We have $o(xh)=lcm(o(x),o(h))$ and  
$o(yh)= lcm(o(y),o(h))$.
It follows that
$lcm(d_p, e_h)\mid o(yh)$.
Define $f$ from $xC_{d/p}$ onto $yC_{d/p}$ by
$f(xh)=yh$.
By above argument $lcm(d_p, e_h)\mid o(f(xh))$
for all $h\in C_{d/p}$.
Since $lcm(d_p, e_h)\mid o(f(xh)^{o(f(xh)C_{d})})$ for any $h\in C_{d/p}$, we deduce that
 $o(t)\mid o(f(\beta(t)^{o(f(\beta(t))C_{d})})$ for all $t\in T$.

Clearly,   if $A$ is a non-empty subset of $C_{d/p}$, the  the restriction of $f$ to $xA$ is a bijection from
 $xA$ to  $yA$.
\end{proof}

Now we are ready to prove the following theorems which generalize the Question \ref{is}.
   \begin{thm}  \label{main3341}
  Let $G$ be a finite group of order $n$, and let $d$   a divisor of $exp(G)$. Let $d=q_1^{\delta_1}\ldots q_v^{\delta_v}$ where $q_1<\ldots<q_v$ are prime numbers,  and let $B$ be a non-empty subset of $B_{p^{\alpha}}$ where $p$ is a prime number such that  $p^{\alpha}\neq 1$ and $p<q_1$ and $[a]\subseteq B$ for any $a\in B$.
   Then  there exist a subset $X$ of a  transversal for $C_{d}$ in $C_n$ and a bijection $f$ from $Sol(B,d,G)$ onto $XC_{d}$ such that 
  For  all $z\in Sol(B,d,G)$, we have $$o(z)\mid o((f(z))^{o(f(z)C_{d\cdot p^{\alpha}})}).$$

\end{thm}
  \begin{proof}
   
 Let    $B=[a_1]^G\cup\ldots\cup [a_t]^G$ where $[a_i^G]\cap [a_j^G]=\varnothing$ for all $1\leq i\neq j\leq t$. We proceed by induction on $t$.
 First suppose that $t=1$.  Set $a_1=a$.
 We proceed by induction on $n$.
  If $n=o(a)$, then the proof is trivial. So suppose that for all groups of order less than $n$, the theorem is true.
  Now we proceed by induction on $d$. First suppose that $d=1$. 
  By Frobenius Theorem $|Sol(1,p^{\alpha-1}\cdot n_{p'},G)|\geq p^{\alpha-1}\cdot n_{p'}$.
  Hence $|Sol(B,1,G)|\leq n- p^{\alpha-1}\cdot n_{p'}$. Since $|C_n\setminus C_{ p^{\alpha-1}\cdot n_{p'}}|\geq n- p^{\alpha-1}\cdot n_{p'}$, there exists a subset $X$ of  $C_n\setminus C_{ p^{\alpha-1}\cdot n_{p'}}$ of size $|Sol(B,1,G)|$.
  Let  $f$ be any bijection from $Sol(B,1,G)$ onto $X=XC_{1}$. Then $f$ satisfies  the conditions of the theorem. 
  So suppose that $d>1$. By induction hypothesis,   the theorem is true whenever  $t<d$ is a divisor of $exp(G)$. 
  We consider the following two Cases:

  {\bf Case 1.}
 First suppose that $a\not\in Z(G)$. Let   $\overline{d}=gcd(|C_G(a)|,d)$.
 For any $a^g\in a^G$, let $S_{a^g}(\overline{d})=Sol([a^g],\overline{d},C_G(a^g))$.
 From Proposition \ref{dis} (iv)  $S_{a^g}(\overline{d})\cap S_{a^h}(\overline{d})=\varnothing$ whenever $[a^g]\neq  [a^h]$.
 By induction hypothesis, for any $a^g\in a^G$, there exist a subset $X_{a^g}$ of a   transversal of $C_{\overline{d}}$ in $C_{n}$ and a bijection $\rho_{a^g}$ from $ S_{a^g}(\overline{d})$ onto $X_{a^g}C_{\overline{d}}$ such that for all $z\in S_{a^g}(\overline{d})$, we have $o(z)\mid o((\rho_a(z))^{o(\rho_a(z)C_{p^{\alpha}\cdot \overline{d}})})$. 
Let $[a^G]=\{u_1,\ldots,u_r\}$.
 Since 
$$|Sol([a],d,G)|\leq n-n_{p'}\cdot p^{\alpha-1}=|\{x\in C_n: p^{\alpha}\mid o(x)\}|,$$    there exist a   transversal $U$ for 
$C_{\overline{d}}$ in $C_n$ and the disjoint subsets   $X_{u_1},\ldots,X_{u_1}$ of $U$ such that
for any $1\leq i\leq r$, there exists   a bijection  $\rho_{u_i}$   from $S_{u_i}(d_a)$ onto $X_{u_i}C_{\overline{d}}$ such that $o(z)\mid o(\rho_{u_i}(z)^{o(\rho_{u_i}(z)C_{p^{\alpha}\cdot \overline{d}})})$ for all $z\in S_{u_i}(\overline{d})$. From Proposition \ref{dis} (v) $  d\mid |Sol([a],d,G)|$. So $d/\overline{d}\mid |X_{u_1}\cup \ldots\cup X_{u_r}|$. 
Let $y\in (X_{u_1}\cup\ldots \cup X_{u_1})C_d\setminus (X_{u_1}\cup\ldots\cup X_{u_1})C_{\overline{d}}$ and let $x\in (X_{u_1}\cup\ldots \cup X_{u_1})C_{\overline{d}}.$
Then $yC_{\overline{d}}\cap (X_{u_1}\cup\ldots \cup X_{u_1})C_{\overline{d}}=\varnothing.$ By Lemma \ref{jai}, we may
replace all elements of $xC_{\overline{d}}$ with all elements of $yC_{\overline{d}}$ in $(X_{u_1}\cup\ldots \cup X_{u_1})C_{\overline{d}}$. Hence we may 
assume that $X_{u_1}C_{\overline{d}}\cup \ldots\cup X_{{u_r}}C_{\overline{d}}=R C_{d}$ for some subset $R$ of $X$. 
Let $f=\bigcup_{i=1}^r \rho_{u_i}$.
Clearly, $f$  is a bijection from $Sol([a],d,G)$ onto $R C_{d}$ such that   
$o(z)\mid o((f_d(z))^{o(f_d(z)C_{p^{\alpha}\cdot d})})$ for all $z\in Sol([a],d,G)$.

    {\bf Case 2.} 
Suppose that $a\in Z(G)$. Let $C\in C_n$ of order $p^{\alpha}$.
 By induction hypothesis, there exist  a   tranversal   $\frac{R}{\langle c\rangle}$ for $\frac{C_{p^{\alpha}\cdot d_{q_1'}}}{\langle c\rangle}$ in $\frac{C_n}{\langle c\rangle}$ and a bijection $\theta$ from  
 $Sol(B_{q_1^{\delta_1}},d_{q_1'},\frac{G}{\langle a\rangle)})$ onto $\frac{RC_{p^{\alpha}\cdot d_{q_1'}}}{\langle c\rangle}$ such that
 $$o(z)\mid o((\theta(z))^{o(\theta(z)\frac{C_{p^{\alpha}\cdot d}}{\langle c\rangle})}),$$
 for all $z\in Sol(B_{q_1^{\delta_1}},d,\frac{G}{\langle a\rangle)})$.
 Hence there exists a bijection $\mu$ from 
 $Sol([a],d,G)\setminus Sol([a],m,G)$ onto 
 $RC_{q_1'(d)}$ such that
 $$o(z)\mid o((\mu(z))^{o(\mu(z)C_{p^{\alpha}\cdot d})}),$$ for all $z\in Sol([a],d,G)\setminus Sol([a],m,G)$.
 Let $m:=d/q_1=p_1^{\alpha_1}\ldots p_{k}^{\alpha_k}$ where $p_1<\ldots<p_k$ are prime numbers. By induction hypothesis, there exist a subset $Y$ of a transversal for $C_m$ in $C_n$ and a bijection $\beta$ from $Sol([a],m,G)$ onto $YC_m$ such that 
 $o(z)\mid o(f(z)^{o(f(z) C_{p^{\alpha}\cdot m})})$ for all $z\in Sol([a],m,G)$.
 By Proposition  \ref{dis} (v) 
$ m\mid |Sol([a],d,G)\setminus Sol([a],m,G)|$. 
Hence from Lemma \ref{jai}, we may assume that 
 $RC_{d}=TC_{m}$ where $T$ is a subset of a transversal for $C_{m}$ in $C_n$.
 
Since 
$$|Sol([a],d,G)|\leq n-n_{p'}\cdot p^{\alpha-1}=|\{x\in C_n: p^{\alpha}\mid o(x)\}|,$$   
from Lemma \ref{jai}, we may assume that $T\cap Y=\varnothing$.
Also, since $d\mid |Sol([a],d,G)|$, we may assume that $(Y\cup T)C_{m}=HC_{d}$  where $H$ is a subset of a transversal for $C_{d}$ in $C_n$.
Then $\mu\cup \beta$  is a bijection from $Sol([a],d,G)$ onto $HC_{d}$ with desired properties.

 Now, suppose that $t>1$.  
   Let $E=[a_1^G]\cup\ldots\cup [a_{t-1}^G]$.
 By induction hypothesis, there 
  exists a bijection $f_{1,d}$ from $Sol(E,d,G)$ onto $X_1C_{d}$ such that 
$o(z)\mid o((f_{1,d}(z))^{o(f_{1,d}(z)C_{p^{\alpha}\cdot m})})$
for all $z\in Sol(E,d,G)$  
 and  satisfy the conditions of the theorem. Suppose that $E\neq B$.
By Frobenius Theorem $| Sol(1,p^{\alpha-1}\cdot d,G)|\geq n_{p'})\cdot p^{\alpha-1}$, so 
 
  $$| Sol(B,d,G)|\leq n-n_{p'}\cdot p^{\alpha-1}=|C_{n_{p'}}(C_{n_p}\setminus C_{p^{\alpha-1}})|.$$
It follows that
\begin{eqnarray*}
 |Sol(B,d,G)\setminus Sol([a_1^G],d,G)|&\leq&|(C_{n_p}\setminus C_{n_{p'}\cdot p^{\alpha-1}}|-|X_1C_{d}|.
 \end{eqnarray*}
By induction hypothesis, there exist subset $X_2$ of a   transversal for $C_{d}$ in $C_n$ and  a         bijection $f_{2,d}$ from $Sol([a_t],d,G)$ onto $X_2C_{d}$  which satisfy the conditions of the theorem.

 Let $z\in Sol([a_t],d,G)$ such that $f_{2,d}(z)C_{d}\subset X_1C_{d}$.  Since
  \begin{eqnarray*}
| Sol(B,d,G)\setminus Sol([a_1^G],d,G)|&\leq& |(C_{n_{p}}\setminus C_{n_{p'}\cdot p^{\alpha-1}}|-|X_1C_{d}|,
 \end{eqnarray*}
there exists an  element  $u\in (C_{n_p}\setminus C_{n_{p'}\cdot p^{\alpha-1}})\setminus (X_2\cup X_1)C_{d}$.
By Lemma \ref{jai}, we may replace all elements of $f_{2,d}(z)C_{d}$ with all elements of  $uC_{d}$
  in $X_2C_{d}$. Therefore we may assume that
  $X_1\cap X_2=\varnothing$.
  Then $f_d=f_{1,d}\cup f_{2,d}$ is a bijection from $Sol(B,d,G)$ onto  $(X_{1}\cup X_2)C_{d}$ with the desired properties.

  \end{proof}
  Given two partially ordered sets $A $ and $B$, the lexicographical order on the Cartesian product $A\times B$ is defined as
$ (a,b)\leq  (a',b')$  if and only if   $a<a'$  or $a=a'$ and $b\leq b'$.
 Let $\mathbb{Z}_{+}$ be the set of all integers $n\geq 0$, and let  $\mathbb{Z}_{+}^2=\{(a,b): a,b \in \mathbb{Z}_{+}\}$.  
Let  $\succeq$ be the lexicographical order on the Cartesian product $\mathbb{Z}_{+}^2$. Hence 
  $(a_1,b_1)\succeq (a_2,b_2)$ whenever $a_1>a_2$ or $a_1=a_2$ and $b_1>b_2$.

 Let  $d$ be a   divisor of the integer number $n=\rho_1^{\theta_1}\ldots\rho_e^{\theta_e}$ where $\rho_1<\ldots<\rho_e$ are primes. 
For shorten the equations in the next theorem, we use the following notation:  For any divisor $d$ of $n$ and any $1\leq i\leq e$, the divisor  $\frac{d}{gcd(d,\rho_1^{\theta_1}\ldots\rho_i^{\theta_i})}=d_{(\rho_1\rho_2...\rho_i)'}$ of $d$ is denoted by $\lambda_i(d,n)$, also if there is no any ambiguity,  $\lambda_i(d,n)$ simply  is denoted   by $\lambda_i(d)$.
  Now, we are ready to prove the second   result.

     \begin{thm}  \label{main33}
  Let $G$ be a finite group of order $n$, and let $d$   a divisor of $exp(G)$. Let $d=q_1^{\delta_1}\ldots q_v^{\delta_v}$ where $q_1<\ldots<q_v$ are prime numbers, and let   $X_{0,0}=\{1\}$.
Then 
  for all pairs $(i,j)$ with $1\leq i\leq v$ and $0\leq j\leq \alpha_i-1$, there exist a subset $X_{q_i,\alpha_i-j}$ of a   transversal for $C_{\lambda_i(d)}$ in $C_n$ and a bijection $f$ from $Sol(1,d,G)$ onto $\bigcup_{i=1}^k\bigcup_{j=0}^{\delta_i-1}X_{\{q_i,\delta_i-j\}} C_{\lambda_i(d)}\cup X_{0,0}$ such that 
   
   (1) For  all $z\in Sol(B_{q_i^{\delta_i-j}},\lambda_i(d),G)$, we have $$o(z)\mid o(f(z)^{o(f(z)C_{\lambda_i(d) \cdot q_i^{\delta_i-j}})}).$$
  
  (2)   If      $| Sol(B_{q_i^{\delta_i-j}},\lambda_i(d),G)|\geq\varphi(q_i^{\delta_i-j})\cdot \lambda_i(n),$ then $C_{\lambda_i(n)}(C_{q_i^{\delta_i-j}}\setminus C_{q_i^{\delta_i-j-1}})\subseteq X_{q_i,\delta_i-j}C_{\lambda_i(d)}$.

(3)   If     $| Sol(B_{q_i^{\delta_i-j}},\lambda_i(d),G)|\leq\varphi(q_i^{\delta_i-j})\cdot \lambda_i(n),$ then $X_{q_i,\delta_i-j}C_{\lambda_i(d)}\subseteq  C_{\lambda_i(n)}(C_{q_i^{\delta_i-j}}\setminus C_{q_i^{\delta_i-j-1}})$.

\end{thm}
  \begin{proof}
  We proceed by double induction on $n$ and $d$.
 The cases $n=d=1$,  is trivial. If $d=1$, then for any finite group $H$ of order $t$, $f$ from $Sol(1,1,H)$ to $C_1$ is the desired bijection. So suppose that the  theorem is true for any group $H$ and divisor $r$ of $exp(H)$ with $|H|<n$ or $|H|=|G|$ and $r<d$.   Let $m:=d/q_1=p_1^{\alpha_1}\ldots p_{k}^{\alpha_k}$ where $p_1<\ldots<p_k$ are prime numbers.

Let   $M=\bigcup_{i=1}^k\bigcup_{j=0}^{\alpha_i-1} X_{\{p_i,\alpha_i-j\}} C_{\lambda_i(m)}$ and  $f=\bigcup_{i=1}^k\bigcup_{j=0}^{\alpha_i-1}f_{i,\alpha_i-j}$ where $f_{i,\alpha_i-j}$ is a bijection from 
$Sol(B_{p_i^{\alpha_i-j}},\lambda_i(m),G)$ onto $X_{p_i,\alpha_i-j}C_{\lambda_i(m)}$ such that $f_{i,\alpha_i-j}$ and the sets $X_{p_i,\alpha_i-j}C_{\lambda_i(m)}$  satisfy  the conditions (1),(2) and (3).
By induction hypothesis, we may assume that the restriction of $f$ to $M$ is a bijection from $M$ to $f(M)$.
Let $1\leq i\leq k$, and let $$x\in  X_{\{p_i,\alpha_i-j\}} C_{\lambda_i(m)}\setminus (C_{ \lambda_i(m) \cdot p_i^{\alpha_i-j}}\setminus C_{ \lambda_i(m) \cdot p_i^{\alpha_i-j-1}})$$ and $y\in C_n\setminus M$ such that $p_i^{\alpha_i-j}\mid o(y)$. Let $\beta$ be the bijection of
$xC_{\lambda_i(m)}$ onto  $yC_{\lambda_i(m)}$ obtained from Lemma \ref{jai}. 
Let $M_i=(f(Sol(B_{p_i^{\theta_i-j}},\lambda_i(m),G))\setminus xC_{\lambda_i(m)})\cup yC_{\lambda_i(m)}$, and  define a new bijection  $\eta_i$   from 
$Sol(B_{p_i^{\theta_i-j}},\lambda_i(m),G)$ onto $M_i$ by 

$$
\eta_i(z) =
\left\{
	\begin{array}{ll}
		f(z)  & \mbox{if } f(z)\in  M_i \\
		\beta (f(z))  & otherwise.
	\end{array}
\right.
$$

 From here onward whenever we say `` we replace $xC_{\lambda_i(m)}$ with $yC_{\lambda_i(m)}$ in $M$"  it means we do the above process and find $M_i$ and $\eta_i$, and identify    them by $M$ and   $f$, respectively.
First we claim that we may replace some   elements of $M$ by some elements of $C_n\setminus M$  such that all conditions of the theorem   and  the following inequality remain true:
\begin{equation*} 
|Sol(1,d,G)\setminus Sol(1,m,G)|\leq |\{x\in C_n\setminus M : q_1^{\delta_1}\mid o(x)\}|.
\end{equation*}
 
 Let $Z^2=\{(p_r,s): r=1,2,\ldots,k \  s=0,1,\ldots,\alpha_r-1\}$.
 Let $Z^2_2$ be the set of all $(p_r,s)\in Z^2$ such that
 $q_1^{\delta_1}\mid exp(X_{\{p_r,\alpha_r-s\}}C_{\lambda_r(m)})$. If $Z^2_2$ is an empty set, then for all $y\in M$ and   $x\in C_n\setminus M$, we have $o(yC_{d/q_1})_{q_1}\leq o(xC_{d/q_1})_{q_1}$, and so by Lemma \ref{ja}, 
\begin{equation*} 
|Sol(1,d,G)\setminus Sol(1,m,G)|\leq |\{x\in C_n\setminus M : q_1^{\delta_1}\mid o(x)\}|.
\end{equation*}
 So suppose that  $Z^2_+$ is not an empty set. 
Let $(p_i,\alpha_i-j)$ be the maximal element of $Z^2_+$ with respect to the  lexicographic order $\succeq$.
  Hence 
 $q_1^{\delta_1}\nmid exp(X_{\{p_r,\alpha_r-s\}}C_{\lambda_r(m)})$
 for all $r>i$ or $r=i$ and $s>j$.
  The set of all elements of $C_n$ which $p_i^{\alpha_i-j}\mid o(x)$ but $q_1^{\delta_1}\nmid o(x)$ is equal to the set $$C_{n_{(p_i\cdot q_1)'} q_1^{\delta_1-1}}(C_{n_{p_i}}\setminus C_{p_i^{\alpha_i-j-1}})=C_{n_{q_1'}q_1^{\delta_1-1}}\setminus C_{n_{(p_i\cdot q_1)'}}C_{ p_i^{\alpha_i-j-1}}.$$
For any $x \in X_{\{p_i,\alpha_i-j\}}C_{\lambda_i(m)}$ with $q_1^{\delta_1}\mid o(x)$, if there exists $y\in (C_{n_{q_1'}q_1^{\delta_1-1}}\setminus C_{n_{(p_iq_1)'}}C_{ p_i^{\alpha_i-j-1}})\setminus M$, then we may replace $xC_{\lambda_i(m)}$ with $yC_{\lambda_i(m)}$ in $X_{\{p_i,\alpha_i-j\}}C_{\lambda_i(m)}$.
After repeat this process  at most  $w\leq |X_{\{p_i,\alpha_i-j\}}C_{\lambda_i(m)}|$ times, 
we have one of the following two cases:

{\bf Case 1.}
If $q_1^{\delta_1}\mid exp(X_{\{p_i,\alpha_i-j\}}C_{\lambda_i(m)})$, then $C_{n_{q_1'}q_1^{\delta_1-1}}\setminus C_{n_{(p_iq_1)'})}C_{ p_i^{\alpha_i-j-1}}\subseteq M$.
We have $$(C_n\setminus M)\setminus (C_{n_{(p_iq_1)'}}C_{ p_i^{\alpha_i-j-1}}\setminus M)=\{x\in C_n\setminus M: q_1^{\delta_1}\mid o(x)\}.$$
Since $(C_{n_{(p_iq_1)'}}C_{ p_i^{\alpha_i-j-1}}\setminus M)\subseteq C_n\setminus M$,
and $M\subseteq C_n$, we deduce that 
$$(n- |M|)- |C_{n_{(p_iq_1)'}}C_{ p_i^{\alpha_i-j-1}}\setminus M|=|\{x\in C_n\setminus M: q_1^{\delta_1}\mid o(x)\}|.$$

Let $n_{q_1'}p_i^{\alpha-j-1}=\varrho_1^{\omega_1}\ldots  \varrho_b^{\omega_b}$ where $\varrho_1<\ldots  <\varrho_b$ are prime numbers. 
 Let $\Phi$ be the set of all $(r, \omega_r-s)$ where
  $1\leq r\leq b$ and $0\leq s\leq \omega_r-1$, and let $\varsigma$ be the set of all $(r, \omega_r-s)\in \Phi$ such that $\varrho_r^{\omega_r-s}\mid m$.
  Also, let $\Upsilon$ be the set of all $(r, \omega_r-s)\in \varsigma$ such that $|Sol(B_{\varrho_r^{\omega_r-s}},\lambda_r(m),G)|< \lambda_r(n)\varphi(\varrho_r^{\omega_r-s}).$

 Let $ (r, \omega_r-s)\in\varsigma\setminus \Psi$.
Then $ |Sol(B_{\varrho_r^{\omega_r-s}},\lambda_r(m),G)|\geq \lambda_r(n)\varphi(\varrho_r^{\omega_r-s}),$   
so by our assumption  $C_{\lambda_r(n)}(C_{\varrho_r^{\omega_r-s}}\setminus C_{\varrho_r^{\omega_r-s-1}})\subseteq X_{\varrho_r,\omega_r-s}C_{\lambda_r(m)},$  hence
 \begin{eqnarray*}
|Sol(B_{\varrho_r^{\omega_r-s}},\lambda_r(n),G)\setminus  Sol(B_{\varrho_r^{\omega_r-s}},\lambda_r(m),G)|&\geq&|C_{\lambda_r(n)}(C_{\varrho_r^{\omega_r-s}}\setminus C_{\varrho_r^{\omega_r-s-1}})\setminus f(Sol(B_{\varrho_r^{\omega_r-s}},\lambda_r(m),G))|\\&=&|C_{\lambda_r(n)}(C_{\varrho_r^{\omega_r-s}}\setminus C_{\varrho_r^{\omega_r-s-1}})\setminus X_{\varrho_r,\omega_r-s}C_{\lambda_r(m)}|\\&=&0.
\end{eqnarray*}

Let $(r, \omega_r-s)\in \Upsilon$. By our assumption  $ X_{\varrho_r,\omega_r-s}C_{\lambda_r(m)}\subseteq C_{\lambda_r(n)}(C_{\varrho_r^{\omega_r-s}}\setminus C_{\varrho_r^{\omega_r-s-1}})$.
From Proposition \ref{dis} (v) $|Sol(B_{\varrho_r^{\omega_r-s}},\lambda_r(n),G)|\geq  \lambda_r(n)\varphi(\varrho_r^{\omega_r-s}),$
so
\begin{eqnarray*}
&&|Sol(B_{\varrho_r^{\omega_r-s}},\lambda_r(n),G)\setminus  Sol(B_{\varrho_r^{\omega_r-s}},\lambda_r(m),G)|\\&\geq&|C_{\lambda_r(n)}(C_{\varrho_r^{\omega_r-s}}\setminus C_{\varrho_r^{\omega_r-s-1}})\setminus f(Sol(B_{\varrho_r^{\omega_r-s}},\lambda_r(m),G))|\\&=&
|C_{\lambda_r(n)}(C_{\varrho_r^{\omega_r-s}}\setminus C_{\varrho_r^{\omega_r-s-1}})\setminus X_{\varrho_r,\omega_r-s}C_{\lambda_r(m)}|.
\end{eqnarray*}

If $(r, \omega_r-s)\in \phi\setminus \varsigma$, then  from Proposition \ref{dis} (v) $|Sol(B_{\varrho_r^{\omega_r-s}},\lambda_r(n),G)|\geq  \lambda_r(n)\varphi(\varrho_r^{\omega_r-s}),$
so,  
\begin{eqnarray*}
|Sol(B_{\varrho_r^{\omega_r-s}},\lambda_r(n),G)\setminus M|&=&|Sol(B_{\varrho_r^{\omega_r-s}},\lambda_r(n),G)|\\&\geq& \varphi(\varrho_r^{\omega_r-s})\cdot \lambda_r(n)\\&=& |C_{\lambda_r(n)}(C_{\varrho_r^{\omega_r-s}}\setminus C_{\varrho_r^{\omega_r-s-1}})|.
\end{eqnarray*}
Therefore 
\begin{eqnarray*}
&&|C_{p_i^{\alpha-j-1}q_1'(p_i'(n))}\setminus M|\\&=&
 |\bigcup_{(r,\omega_r-s)\in\Phi}C_{\lambda_r(n)}(C_{\varrho_r^{\omega_r-s}}\setminus C_{\varrho_r^{\omega_r-s-1}})\setminus M|\\&=& |\bigcup_{(r,\omega_r-s)\in\Upsilon}C_{\lambda_r(n)}(C_{\varrho_r^{\omega_r-s}}\setminus C_{\varrho_r^{\omega_r-s-1}})\setminus M|+|\bigcup_{(r,\omega_r-s)\in\Phi\setminus \varsigma}C_{\lambda_r(n)}(C_{\varrho_r^{\omega_r-s}}\setminus C_{\varrho_r^{\omega_r-s-1}})\setminus M|\\&\leq& \sum_{(r,\omega_r-s)\in\Upsilon}|C_{\lambda_r(n)}( C_{\varrho_r^{\omega_r-s}}\setminus C_{\varrho_r^{\omega_r-s-1}})\setminus f(Sol(B_{\varrho_r^{\omega_r-s}},\lambda_r(m),G))|\\&+& \sum_{(r,\omega_r-s)\in\Phi\setminus \varsigma}|C_{\lambda_r(n)}(C_{\varrho_r^{\omega_r-s}}\setminus C_{\varrho_r^{\omega_r-s-1}})\setminus M|\\&\leq&  
 \sum_{(r,\omega_r-s)\in\Phi}|Sol(B_{\varrho_r^{\omega_r-s}},\lambda_r(n),G)\setminus  Sol(B_{\varrho_r^{\omega_r-s}},\lambda_r(m),G)|.
\end{eqnarray*}
Let $(r,\gamma_r-s)\in \Phi$, and let $x\in Sol(B_{\varrho_r^{\omega_r-s}},\lambda_r(n),G)\setminus  Sol(B_{\varrho_r^{\omega_r-s}},\lambda_r(m),G)$.
Then $o(x)\nmid d$, and so $x\not\in Sol(1,d,G)$.
  It follows that
\begin{eqnarray*}
|Sol(1,d,G)\setminus Sol(1,m,G)|&\leq&
 |Sol(1,n,G)\setminus Sol(1,m,G)|\\&-&
 |\sum_{(r,\gamma_r-s)\in \Phi}|Sol(B_{\varrho_r^{\omega_r-s}},\lambda_r(n),G)\setminus  Sol(B_{\varrho_r^{\omega_r-s}},\lambda_r(m),G)|\\&\leq& n-|M|-|C_{p_i^{\alpha-j-1}n_{(p_iq_1)'}}\setminus M|\\&=&
 |\{x\in C_n\setminus M: q_1^{\delta_1}\mid o(x)\}|.
 \end{eqnarray*}

{\bf Case 2.} $q_1^{\delta_1}\nmid exp(X_{\{p_i,\alpha_i-j\}}C_{\lambda_i(m)})$. 
If there exists $(p_r,\alpha_r-s)\prec (p_i,\alpha_i-j)$ such that 
$q_1^{\delta_1}\mid exp(X_{\{p_r,\alpha_r-s\}}C_{\lambda_r(m)})$, then we repeat the above process for $(p_r,\alpha_r-s)$.
In the end, we have    
\begin{equation*} 
|Sol(1,d,G)\setminus Sol(1,m,G)|\leq |\{x\in C_n\setminus M : q_1^{\delta_1}\mid o(x)\}|,
\end{equation*}
as claimed.

Since $d_{q_1'}<d$ by Theorem \ref{main3341}, there exist a subset $X_{q_1,\delta_1}$ of a transversal for $C_{d_{q_1'}}$ in $C_n$ and a bijection $f_{1,\delta_1}$   from 
$Sol(B_{q_1^{\delta_1}},d_{q_1'},G)$ onto 
$X_{q_1,\delta_1}C_{d_{q_1'}}$ such that
for any $z\in Sol(B_{q_1^{\delta_1}},d_{q_1'},G)$, we have 
$o(z)\mid o(f_{1,\delta_1}(z)^{o(f_{1,\delta_1}(z)C_{d})})$.
Since \begin{equation*} 
|Sol(1,d,G)\setminus Sol(1,m,G)|\leq |\{x\in C_n\setminus M : q_1^{\delta_1}\mid o(x)\}|.
\end{equation*}
  by  Lemma \ref{jai}, we may assume that $M\cap X_{q_1,\delta_1}C_{d_{q_1'}}=\varnothing.$
Hence $\tau=f\cup f_{1,\delta_1}$  is our bijection with desired properties.

Let $x\in X_{q_1,\delta_1}C_{d_{q_1'}}$ and $y\in C_n\setminus X_{q_1,\delta_1}C_{d_{q_1'}}\cup M$ such that $x\not\in C_{\lambda_1(n)}(C_{q_1^{\delta_1 }}\setminus C_{q_1^{\delta_1-1}})$  and $y\in C_{\lambda_1(n)}(C_{q_1^{\delta_1 }}\setminus C_{q_1^{\delta_1-1}})$.
From  Lemma \ref{jai}, we may replace all elements of $xC_{d_{q_1'}}$ with  all elements of $yC_{d_{q_1'}}$. Therefore, the proofs of cases (2) and (3) are clear.

  \end{proof}

  The following corollary gives affirmative answer to Question \ref{is}.

 \begin{cor} 
  Let $G$ be a finite group of order $n$.
  Then there exists a bijection $f$ from $G$ onto $C_n$ such that $o(x)$ divides $o(f(x))$ for all $x\in G$.
\end{cor} 
\begin{proof}
Let $d:=exp(G)=q_1^{\delta_1}\ldots q_v^{\delta_v}$ where $q_1<\ldots<q_v$ are prime numbers, and let $X_{0,0}=\{1\}$. From Theorem \ref{main33}, 
  for all pairs $(i,j)$ with $1\leq i\leq v$ and $0\leq j\leq \alpha_i-1$, there exist a subset $X_{q_i,\alpha_i-j}$ of a   transversal for $C_{\lambda_i(d)}$ in $C_n$ and a bijection $f$ from $Sol(1,d,G)=G$ onto $\bigcup_{i=1}^k\bigcup_{j=0}^{\delta_i-1}X_{\{q_i,\delta_i-j\}} C_{\lambda_i(d)}\cup X_{0,0}=C_n$ such that 
   for  all $z\in Sol(B_{q_i^{\delta_i-j}},\lambda_i(d),G)$, we have $$o(z)\mid o(f(z)^{o(f(z)C_{\lambda_i(d) \cdot q_i^{\delta_i-j}})}).$$
 
 Let $z\in G\setminus \{1\}$. Then  $z\in Sol(B_{q_i^{\delta_i-j}},\lambda_i(d),G)$ for some $1\leq i\leq v$ and some $0\leq j\leq \delta_i-1$. Hence $$o(o(f(z)^{o(f(z)C_{\lambda_i(d) \cdot q_i^{\delta_i-j}})}))=\frac{o(f(z))}{gcd(o(f(z)),o(f(z)C_{\lambda_i(d) \cdot q_i^{\delta_i-j}}))}\mid o(f(z)).$$
 So $f$ is a bijection from $G$ onto $C_n$ such that 
 $o(z) \mid o(f(z))$ for all $z\in G$.
\end{proof}

\section{Some applications}

 \begin{rem}
     Let $B(C_n)$ be the set of all subgroups of the cyclic group $C_n$ of order $n$, and let 
  $\tau(C_n)$ be the set of all subsets of $B(C_n)$. Then $\tau(C_n)$ is a topology on $C_n$ where $B(C_n)$ is  a countable  base for $\tau(C_n)$. 
Let $G$ be a group of order $n$, and $f$ be the bijection in Question \ref{is}. Then $f^{-1}(\tau(C_n))$ is a topology on $G$ which is induced by $\tau(C_n)$. 
Hence affirmative answer to Question \ref{is} gives us a non-trivial discrete topology $\tau(G)$ in any finite group $G$.
 \end{rem}

If $G$ is a finite group, let $f(G)$ be the average order of an element of $G$, i.e.,
 $f(G) = \frac{1}{|G|}\sum_{g\in G}o(g)$.
 It is easy to see that $f(G)\leq h(G)$, where $h(G)$ is number of conjugacy classes of $G$.
 It is known by Lindsey's Theorem \cite{Lin} that  $f(G)\leq f(C_n)$ for any finite group $G$ of order $n$.
 The  affirmative answer to Question \ref{is}, leads us to the new proof of this result.

Let $f$ be a  function from $\mathbb{R}$   to $\mathbb{R}$ and $P_k(G)$ be the set of all subsets of $G$ of size $k$. Recall that $\psi_{f,k}(G) =\sum\limits_{S\in P_k(G)}\mu(S)$ where $G$ is a finite group and $\mu(S)=\Pi_{x\in S}f(o(x))$.  Also, we may define  $\psi^{f,k}(G) =\sum\limits_{x\in G} (f(o(x)))^k$.
 Similar to the case where $ l = 1$ can simply be proved: If  $f$   is an increasing  function and $ G = A \times B $ is a finite group, then  $ \psi^{f,l}(G) \leq \psi^{f,l}(A)\cdot \psi^{f,l}(B) $ and  $ \psi^{f,l}(G) = \psi^{f,l}(A) \cdot \psi^{f,l}(B) $ if and only  $ gcd (| A |, | B |) = 1$.

Let $f$ be a  function from $\mathbb{R}$   to $\mathbb{R}$ and $P_k(G)$ be the set of all subsets of $G$ of size $k$. Define $\psi_{f,k}(G) =\sum\limits_{S\in P_k(G)}\mu(S)$ where $G$ is a finite group and $\mu(S)=\Pi_{x\in S}f(o(x))$.     Positive answer to   Question  \ref{is}, leads to the following result.

\begin{thm}
Let $G$ be a finite non-cyclic group of order $n$ and $k\in \mathbb{N}$.

(a) If $f$   is an increasing  function, then   $\psi_{f,k}(G)<\psi_{f,k}(C_n)$.

 (b) If $f$   is a  decreasing  function, then   $\psi_{f,k}(G)>\psi_{f,k}(C_n)$.

 \end{thm}

 Let $H$ be a finite group, and let $a \geq  2$ be an integer.  Define the undirected
multigraph $G(a, H)$ with vertex set $H$ and $x\sim y$ if $x^a = y$, with an additional edge
if $y^a = x$.
 Let $N(a, H)$ denote the number of connected components in $G(a, H).$ This
graph has connection with algorithmic number theory and cryptography, for example see \cite{4}, \cite{6} and \cite{12}. Let $N(a, H)$ denote the number of connected components
in $G(a, H)$. Positive answer to Question 1.1, confirms the Conjecture 1.3, posed by M.
Larson in \cite{matt}.

\begin{conj}
Let $G$ be a group of order $n.$ Then
$$N(a, G)\geq N(a,C_n).$$

 \end{conj}
\begin{proof} From Theorem \ref{main33}, there exists a bijection $f$ from $G$ onto $C_n$ such that
$o(x)\mid o(f(x))$ for all $x\in G$.
By  the successive comments after Lemma 1 of \cite{matt}, for all finite groups $H$,
\begin{equation}
 N(a, H)=
  \sum\limits_{\substack{g\in H\\
                  gcd(o(g),a)=1}}
       \frac{1}{ord_{o(g)}(a)}
\end{equation}
where  $ord_n(a)$ denote the multiplicative order of
$a$ in $\frac{Z}{nZ}$. Let $g\in G$ such that $gcd(o(g),a)=gcd(o(f(g)),a)=1$, and let $ord_{o(g)}(a)=d$ and $ord_{o(f(g))}(a)=m$.
Since $o(f(g))\mid a^m-1$ and $o(g)\mid o(f(g))$, we have $ o(g)\mid a^m-1$.
Hence $d\mid m$. It follows that $$\frac{1}{ord_{o(g)}(a)}=\frac{1}{d}\geq \frac{1}{m}= \frac{1}{ord_{o(f(g))}(a)}$$
Consequently, from (3.1)
\begin{equation*}
  \sum\limits_{\substack{g\in G\\
                  gcd(o(g),a)=1}}
       \frac{1}{ord_{o(g)}(a)} \geq
  \sum\limits_{\substack{f(g)\in C_n\\
                  gcd(o(f(g)),a)=1}}
       \frac{1}{ord_{o(f(g))}(a)}.
\end{equation*}

\end{proof}

\bibliographystyle{amsplain}

\end{document}